\documentclass{amsart}
\usepackage{amssymb}
\usepackage{latexsym}
\usepackage{color}
\newcommand{\ZZ}{{\mathbb Z}}
\newcommand{\RR}{{\mathbb R}}

\newtheorem{theorem}{Theorem}
\newtheorem{lemma}{Lemma}[section]
\newtheorem{prop}[lemma]{Proposition}
\newtheorem{coro}[lemma]{Corollary}
\newtheorem{definition}[lemma]{Definition}
\newtheorem{remark}[lemma]{Remark}

\newtheorem{example}[theorem]{Example}

\newcommand{\CE}{{\mathcal{E}}}

\newcommand{\CD}{{\mathcal{D}}}

\newcounter{smalllist}

\newcommand{\vL}{\varLambda}

\newcommand{\ts}{\hspace{0.5pt}}

\newcommand{\Hmm}[1]{\leavevmode{\marginpar{\tiny%
$\hbox to 0mm{\hspace*{-0.5mm}$\leftarrow$\hss}%
\vcenter{\vrule depth 0.1mm height 0.1mm width \the\marginparwidth}%
\hbox to 0mm{\hss$\rightarrow$\hspace*{-0.5mm}}$\\\relax\raggedright
#1}}}

\begin{document}
\title[Diffusion on Delone sets]{Diffusion on Delone sets}

\author[]{Sebastian Haeseler$^1$}
\author[]{Xueping Huang$^2$}
\author[]{Daniel Lenz$^3$}
\author[]{Felix Pogorzelski$^4$}

\address{$^1$ Mathematisches Institut, Friedrich Schiller Universit\"at Jena,
  D-07743 Jena, Germany, sebastian.haeseler@uni-jena.de}
\address{$^2$ Department of Mathematics, Nanjing University of Information Science and Technology, 210044 China, \mbox{hxp@nuist.edu.cn} }
\address{$^3$ Mathematisches Institut, Friedrich Schiller Universit\"at Jena,
  D-07743 Jena, Germany, daniel.lenz@uni-jena.de,\\
  URL: http://www.analysis-lenz.uni-jena.de/ }
\address{$^4$ Technion Israel Institute of Technology Haifa,
  IL-32000 Haifa, Israel, \mbox{felixp@technion.ac.il}}


\begin{abstract}
We consider  graphs associated to Delone sets in Euclidean space.
Such graphs arise in various ways from tilings. Here, we  provide a
unified framework. In this context, we study the associated
Laplace operators and show Gaussian heat kernel bounds for their
semigroups. These results apply to both metric and discrete graphs.
\end{abstract}

\maketitle


\section*{Introduction}
In this article we initiate the study of metric graphs associated to
Delone sets in arbitrary dimensions.

Delone sets  are well-spaced  subsets  of Euclidean space in the
sense that their points have a minimal distance to each other and at
the same time do not admit arbitrarily large holes. They have been of
interest in geometry for a long time.  In recent years they have
attracted substantial attention in mathematical physics as they can
serve as models for the atomic positions of solids.

In particular, Delone sets (with additional regularity features)
play a key role in modeling quasicrystals \cite{Lag,LP}. In this
context, both diffraction of Delone sets, see e.g.\@ the survey
\cite{BL}, and the quantum mechanical treatment of models based on
Delone sets 
   {are} of interest, see e.g.\@ the survey \cite{DEG}. The
quantum mechanical treatment is based on the Schroedinger equation.
Accordingly, spectral theory of Schroedinger operators i.e.\@
operators of the form Laplacian plus a potential, has been a main
topic in the investigation of quasicrystals. While the spectral
theory is understood in impressive detail in the one-dimensional
case \cite{DEG}, this is not the case for the higher dimensional
situation. In that case, the operator algebras associated to
quasicrystals have been intensively studied, resulting in a deep
understanding of their K-theory, see the survey  \cite{KP},  as well
as in investigations centered around an averaged quantity called
integrated density of states. Also, some results on generic singular
continuous spectra are known \cite{LS} and  recent years have
witnessed  developments in the study of certain product models on
lattices, \cite{DGS}. However, as far as the spectral theory of
quasicrystal operators  in  Euclidean space of dimension bigger than
one is concerned it seems fair to say that the basic picture is
rather unclear. Even the spectral properties of discrete Laplacians
alone are not well understood on such basic examples as the Penrose
lattice. This may even be seen as one of the most important open
questions in this field.

Here, we look at such Laplace type operators  from a somewhat
different point of view, i.e.\@ from the point of view of diffusion.
It turns out that in this respect the situation is much more
accessible and the operators are very comparable to the usual
Laplacian  on Euclidean space. In fact, our main results,
Theorem~\ref{thm-gaussian} and Theorem~\ref{thm-gaussian-disc}, give
Gaussian estimates for rather general classes of operators
associated to arbitrary Delone sets. This generalizes an earlier
result of Telcs \cite{Tel} for the Penrose tiling. At the same time
this also generalizes parts of the results of Pang \cite{Pang} on
the square lattice.   {(}Note, however, that the main thrust of
\cite{Pang} is on scaling of the square lattice, which is not
addressed in the present paper.)

The first step in our investigation is to associate graphs to Delone
sets. In concrete situations such graphs are already present because
there is a tiling structure which   {either can}  be seen as a graph
structure or    {gives} rise to a graph structure in natural way. To
deal with the general situation, we introduce the concept of
neighbor relations for   {an} arbitrary Delone   {set}. A pair
consisting of a neighbor relation and a Delone set then gives rise
to a graph based on the Delone set. In this construction, the
well-spacing of Delone sets is reflected in the arising graphs being
roughly isometric to $\RR^N$,   {see} Corollary
\ref{Rough-isometry}. All of this is discussed in Section
\ref{sec-Delone}. In Section \ref{sec-Tiling} we then show that
every Delone set admits a canonical neighbor relation. This is based
on a study of the Voronoi construction. There, we also discuss how
tilings and CW-complexes fit into our framework. From a conceptual
point of view  Section \ref{sec-Delone} and Section \ref{sec-Tiling}
are rather relevant as they set up the framework to study, which may
be of interest for further studies as well, and show how earlier
examples fit into the framework.

Our  main  specific results are then discussed in Section
\ref{sec-Metric} and Section \ref{sec-discrete} respectively. More
specifically,  Section \ref{sec-Metric} deals with metric graphs and
Section \ref{sec-discrete} deals with discrete graphs. Both metric
and discrete graphs can be seen as natural candidates for a  study
of diffusion on Delone sets.  While discrete graphs have - at least
implicitly - been around in this context,
the above mentioned
investigation of the operator algebras and  the integrated density
of states for metric graphs associated to Delone sets does not seem
to have been considered before. This is rather remarkable as metric graphs
and their associated operators have attracted attention as models
for various kinds of random operators, see e.g. \cite{AGA,BCFK} for
recent collections dealing with metric graphs and their features in
a variety of cases. As it stands our results are then the first
results on Laplacians on metric graphs associated to aperiodic
order. As far as methods go, for both the case of discrete graphs
and of metric graphs our main result follows from local regularity
features combined with a main theorem from Barlow / Bass / Kumagai
\cite{BBK}. In fact, for metric graphs we have to work a bit harder
to show the necessary local regularity, which is automatically
satisfied in the discrete case.

\bigskip

\textbf{Acknowledgements.} The authors gratefully acknowledge
financial support from DFG. D.L. would also like to thank Peter
Stollmann and Ivan Veseli\'{c} for delightful discussions on quantum
graphs. X. H. was partially supported by The Startup Foundation for
Introducing Talent of NUIST. F.P.\@ was supported in part by
a Technion Fine fellowship.

\section{Delone sets, neighbor relations,  and the associated
graphs}\label{sec-Delone} In this section we introduce the basic
objects of our study. Delone sets are subsets of Euclidean space
with certain uniform spacing properties. We will associate graphs to
Delone sets. The vertices of the graphs will be the points of the
Delone sets. The edges will be defined by an additional piece of
information. To store this information we introduce the concept of
neighbor relations.

\bigskip

Our basic setup is as follows: We consider subsets of Euclidean
space $\RR^N$. The Euclidean norm is denoted by $\| \cdot\|$ and the
closed ball around the origin $0$ with radius $S$ is denoted by
$B_S$ and the open ball around the origin $0$ with radius $S$ by
$U_S$. The Euclidean metric on $\RR^N$ is denoted by $d$, i.e.
$$d(x,y) := \|x- y\|.$$

The  closed line  segment $[x,y]$  between two points $x,y\in \RR^N$
is defined by
$$[x,y] :=\{ x + t (y-x) : t\in [0,1]\}$$
and the open line segment $]x,y[$ between two points $x,y\in\RR^N$
is defined by
$$]x,y[ :=\{ x + t (y-x) : t\in (0,1)\}.$$
 The Lebesgue measure of a measurable subset of $\RR^N$ is denoted
by $|M|$ and the cardinality of a set $F$ is denoted by $\sharp F$.

\medskip

\begin{definition} Let $\vL$ be a subset of $\RR^N$. Then,  $\vL$ is called
  uniformly discrete if there exists $r>0$ with
$$\|x-y\|\geq 2 r$$
 for all
  $x,y\in \vL$ with $x\neq y$.  The set $\vL$ is called relatively
  dense if  there exists  $R>0$ with
$$\RR^N =\bigcup_{x\in \vL} ( x + B_R ).$$
  If $\vL$ is both uniformly discrete (with parameter $r$)  and
  relatively dense (with parameter $R$)  it is
  called a Delone set or an  $(r,R)$ - Delone set.
\end{definition}

\begin{remark} If  $\vL$ is uniformly discrete with parameter $r$, then  open
balls around  points of $\vL$ with radius $r$ are disjoint.  This is
the  reason for the factor $2$ appearing in the above definition.
The largest $r$ with this property is called the packing radius of
$\vL$.  On the other hand if  $\vL$ is relatively dense with
parameter $R$, then no point of $\RR^N$ has distance larger than $R$
to $\vL$. Then, the smallest $R$  with this property is called the
covering radius of $\vL$.
\end{remark}

As mentioned already we need an additional piece of information in
order to define graphs.

\begin{definition}[Neighbor relation]
A neighbor relation  on a Delone set $\vL$ is a subset $\mathcal{N}$
of $\vL\times \vL$ satisfying the following conditions for some
$S>0$:

\begin{itemize}
\item[(N0)] $\mathcal{N}$ is symmetric (i.e. $(x,y)\in \vL$ if and only if
$(y,x)\in\vL$) and contains the diagonal $\{ (x,x) : x\in\vL\}$.
\item[(N1)] $\|x - y\|\leq  S$ for all $(x,y)\in \mathcal{N}$.
\item[(N2)] For arbitrary $x,y\in \vL$  there exists a sequence
$(x_0, \ldots, x_n)$ in $\vL$ with $x_0 = x$, $x_n = y$ and $(x_i,
x_{i+1})\in\mathcal{N}$, $i = 0,\ldots, n$ and
$$ \{x_0,\ldots, x_n\}  \subset [x,y] + B_{S}.$$
\end{itemize}
The number $S$ is called the parameter of the neighbor relation. We
write $x\sim y$ if $(x,y)\in\mathcal{N}$ and then say that  $x$ and
$y$ are neighbors.
\end{definition}

\begin{remark}[Parameters] Note that any $(r,R)$-Delone set is an
$(r',R')$-Delone set as well for any $0 < r' \leq r$ and $R \leq
R'$. Similarly, a neighbor relation to parameter $S$ is a neighbor
relation to any parameter $S'$ with $S' >S$.
\end{remark}

Part (N2) of the definition is crucial. It gives a precise way of
saying that a point in $\vL$ has neighbors in 'all directions'. This
will enable us to show that the resulting graphs are  roughly
isometric to the ambient space.  In the next section we will show
that any Delone set admits a neighbor relation. More specifically,
we will show  that
$$\mathcal{N}_{max} (\vL):=\{ (x,y)\in \vL \times \vL : \|x - y\|\leq
2 R\}$$ is a neighbor relation (with parameter $2R$). By definition
this is the maximal neighbor relation to this parameter in the sense
that it contains any other such  such neighbor relation.

In this section we proceed by developing the theory of neighbor
relations and the associated graphs. In  particular, we will show
how in our setting there is a certain uniformity to all sorts of
quantities.  Let $\mathcal{N}$ be a neighbor relation on $\vL$. We
then define for  each $x\in \vL$ the  \textit{degree} of $x$ as
$$deg (x):= \sharp \{ y \in\vL : y\neq x \mbox{ and }  y\sim x\}.$$

\begin{prop}[Uniform bounds]\label{prop-uniform}
Let $\vL$ be an $(r,R)$-Delone set and $\mathcal{N}$ a neigbor
relation with parameter $S$ on it.
\begin{itemize}

\item[(a)] We have
$ 2r \leq \|x - y\|\leq S$ for all $(x,y)\in \mathcal{N}$ with
$x\neq y$.
\item[(b)] There is an $D\geq 0$ with $\mbox{deg} \leq D$.

\end{itemize}

\end{prop}
\begin{proof} (a)  By $(x,y)\in\mathcal{N}\subset \vL\times \vL$  we have $2 r \leq \|x - y\|\leq
S$.

\smallskip

(b) By (a), we can bound $deg(x)$ by
$$deg (x) \leq \sharp (x + B_{S}) \cap \vL.$$
As $\vL$ is an $(r,R)$ Delone set,  open  balls of radius $r$ around
different points of $\vL$ are disjoint and  we can bound the latter
quantity to obtain
$$  \sharp(x + B_S ) \cap \vL  \leq  \frac{|B_{S + r}|}{|B_r|}.$$
This finishes the proof.
\end{proof}

Consider now a pair $(\vL,\mathcal{N})$ consisting of an
$(r,R)$-Delone set $\vL$ and a neighbor relation $\mathcal{N}$. We
can then associate a combinatorial graph $G_c(\vL,\mathcal{N})$ as
follows: The vertices of $G_c (\vL,\mathcal{N})$ are given by the
elements of $\vL$ and the edges of $G_c (\vL,\mathcal{N})$ are
exactly the pairs $(x,y) \in\mathcal{N}$.  Note that the degree of
an $x\in \vL$ defined above is exactly the degree of the vertex $x$
in the graph $G_c (\vL,\mathcal{N})$. In particular, the graph $G_c
(\vL,\mathcal{N})$ has uniformly bounded degree.  For $G_c
(\vL,\mathcal{N})$ we call a finite sequence $\gamma = (x_0,\ldots,
x_n)$ of vertices a \textit{path of length $n$ (between $x_0$ and
$x_n$)} if $x_i\sim x_{i+1}$, $i = 0,\ldots, n$. Then, by (N2),  the
graph $G_c (\vL,\mathcal{N})$ is connected, i.e.\@ between all $x,y\in
\vL$  there exists a path. As the graph $G_c (\vL,\mathcal{N})$ is
connected we can  define   the \textit{combinatorial metric} $d_c$
on $\vL$ by
$$d_c (x,y):=\inf \{ n : \mbox{ there is path of length $n$ between
$x$ and $y$}\}.$$ Finally, the vertex set $\vL$ of the graph $G_c
(\vL,\mathcal{N})$ is naturally equipped with the $\sigma$-algebra
of all of its subsets and the  canonical measure $\mu_c$ given by
$$\mu_c (S):= \sharp S$$
for $S\subset \vL$. In this way $(G_c (\vL, \mathcal{N}), d_c,
\mu_c)$ is a locally compact, complete  metric space with  measure
$\mu_c$.

\medskip

We can also associate to $(\vL,\mathcal{N})$ the \textit{metric
graph} $G_m (\vL,\mathcal{N}) $, which arises by `gluing' in
intervals between $x,y\in \vL$ with $(x,y) \in \mathcal{N}$. Roughly
speaking we will glue in the intervals of the form   $]x,y[$ and we
will consider intervals between different vertices as disjoint.
Specifically, we define
$$G_m (\vL,\mathcal{N}):= S(\vL, \mathcal{N}) / \sim,$$
with
$$S(\vL,\mathcal{N}) :=\{ (s,(x,y))\in \RR^N\times \mathcal{N}  : s\in
[x,y]\}$$ and the relation $\sim$ given by
$$(s,(x,y))\sim (s,(y,x)) \mbox{ and } (x,(x,y)) \sim (x,(x,z))$$
for all $(x,y),(x,z)\in \mathcal{N}$ and all $s\in [x,y]$. The class
of an element $(s,(x,y))$ is denoted by $[(s,(x,y))]$.

The elements of the form $[(x,(x,y))]$ are then called the
\textit{vertices} of $G_m (\vL,\mathcal{N})$ and the sets
$$e_{(x,y)}:=\{[(s,(x,y))] : s\in ]x,y[\}$$
with $(x,y)\in\mathcal{N}$ are called \textit{metric edges}. The
length of the metric edge $e_{(x,y)}$, denoted by $l(x,y)$, is
defined as
$$l(x,y) :=\|x-y\|.$$
Note that the metric edge $e_{(x,y)}$  is canonically homeomorphic
and in fact isometrically isomorphic to the open interval
$]0,l(x,y)[$ in $\RR$. We can introduce a metric on $G_m
(\vL,\mathcal{N})$ as follows: Let $a,b\in G_m (\vL,\mathcal{N})$ be
given. Then, a sequence $(x_0, \ldots, x_{n+1})$ of vertices is
called an \textit{admissible path} joining $a$ and $b$ if
$(x_j,x_{j+1})\in \mathcal{N}$ for $j=0,\ldots, n$ and $a = [(s,
(x_0,x_1))]$ and $b = [(t,(x_n, x_{n+1})]$ with suitable (unique)
$s,t\in \RR^N$. We then define the \textit{length} of such a path
$l(a,b,x_0,\ldots, x_{n+1})$ as
$$l(a,b,x_0,\ldots, x_{n+1}) := \| s- x_1\|  +\sum_{j=1}^{n-1} |x_{j} -
x_{j+1}| + \|x_n - t\|$$ and set
$$d_m (a,b):=\inf l(a,b,x_0,\ldots, x_{n+1}),$$
where the infimum is taken over all admissible paths joining $a$ and
$b$. Then, it is not hard to see that $d_m$ defines a metric on $G_m
(\vL,\mathcal{N})$. In this way, $G_m (\vL,\mathcal{N})$ becomes a
metric space.  There is a uniform upper bound on the lengths of
these metric edges (by the subsequent Lemma \ref{lem-equvialence}).
Note that we also have a uniform lower bound on the lengths of the
edges as well as an uniform upper bound  on the degree by
Proposition \ref{prop-uniform}.  Given the uniform bounds on the
edge lengths and on the degree it is not hard to see that $G_m
(\vL,\mathcal{N})$ is a complete metric space. Moreover, $G_m
(\vL,\mathcal{N})$ is locally compact. Clearly, the Lebesgue measure
on the edges  induces a measure on the graph $G_m
(\vL,\mathcal{N})$. This measure will be denoted by $\mu_m$.

\begin{remark} Of course, there is a close relationship between the
two graphs $G_c (\vL,\mathcal{N})$ and $G_m (\vL,\mathcal{N})$. One
can think of $G_c (\vL,\mathcal{N})$ as a  combinatorial graph
underlying $G_m (\vL,\mathcal{N})$. Conversely, one can think of
$G_m (\vL,\mathcal{N})$ as arising from $G_c (\vL,\mathcal{N})$ by
'gluing in'  intervals of length $l(e)$ at each edge $e$.
\end{remark}

\medskip

Recall that two metrics $e$ and $e'$  on the same space are called
\textit{equivalent} if there exists a $c>0$ with
$$ \frac{1}{c} e' \leq e \leq c \; e'.$$

\begin{lemma}[Equivalence of distances]\label{lem-equvialence} Let
$\vL$ be a  Delone set with neighbor relation $\mathcal{N}$. The
metrics $d, d_c$ and $d_m$ are all equivalent on $\vL$.
\end{lemma}
\begin{proof} Chose $r,R,S>0$ such that $\vL$ is an $(r,R)$-Delone
set and $S$ is a parameter for $\mathcal{N}$. Let  $x,y\in \vL$  be
given. Then, we clearly have
$$d(x,y) \leq d_m (x,y) \mbox{ and } d_m (x,y) \leq S d_c (x,y).$$
We are now heading to providing an estimate of $d_c (x,y)$ in terms
of $d(x,y)$. Without loss of generality we assume $x\neq y$. By
(N2), there exists  a sequence $(x_0, \ldots, x_n)$ in $\vL$ with
$x_0 = x$, $x_n = y$ and $x_i \sim x_{i+1}$, $i = 0,\ldots, n$ and
$$ \{x_0,\ldots, x_n\} \subset [x,y] + B_{S}.$$
This gives
$$d_c (x,y) \leq n \leq \vL \cap ( [x,y] +
B_{S}).$$ As the open balls with radius $r$ around different points
of $\vL$ are disjoint we can bound the latter term by
$$...  \leq
\frac{ |[x,y] + B_{S+r}|}{|U_r|}.$$ To estimate that last term we
use that
$$|[x,y] + B_{S+r}| = \| y - x\| \sigma_{N} + |B_{S+r}|$$
with the volume $\sigma_N$ of the $(S+r)$-ball in $\RR^{N-1}$. As $\|y -
x\|\geq 2 r$ for $x\neq y$ we find that
$$|[x,y] + B_{S+r}|  \leq \|y - x\| \left(\sigma_N +\frac{1}{2r} |B_{S+r}|\right).$$
Putting this together we arrive at
$$d_c (x,y) \leq C d (x,y)$$
with $C = \frac{ \sigma_N + \frac{1}{2r}|B_{S+r}|}{|U_r|}$.
\end{proof}

Two metric measure spaces $(X_j, e_j, \mu_j)$, $j=1,2$ are called
\textit{roughly isometric} if there exists a map $\phi :
 X_1\longrightarrow X_2$ as well as $\varrho>0$ and $\delta>1$ with

\begin{itemize}

\item $\bigcup_{x\in X_1} B_{e_2}(\phi(x), \varrho) = X_2$,

\item $\frac{1}{\delta}  (e_1 (x,y) - \varrho) \leq e_2 (\phi (x), \phi
(y)) \leq \delta ( e_1 (x,y) + \varrho)$ for all $x,y\in X_1$,

\item $\frac{1}{\delta} \mu_1 (B_{e_1}(x,\varrho)) \leq \mu_2
(B_{e_2}(\phi (x), \varrho)) \leq \delta \mu_1 (B_{e_1}(x,\varrho))$
for all $x\in X_1$,
\end{itemize}
where $B_{e_j} (z, s)$ denotes the ball in $X_j$ around $z$ with
radius $s$ with respect to $e_j$.   {It is well known that rough
isometry is an equivalence relation for metric measure spaces.}

\smallskip

\begin{coro}[Rough isometry]\label{Rough-isometry}
The metric measure spaces $(\RR^N,d,  \lambda)$, $(G_c(\vL,\mathcal{N}), d_c, \mu_c)$
and $(G_m (\vL,\mathcal{N}), d_m, \mu_m)$ are roughly isometric.
\end{coro}
\begin{proof} This follows easily from Lemma \ref{lem-equvialence}.
\end{proof}

\section{Tiling systems and existence of neighbor relations}
\label{sec-Tiling} In this section we  show existence of neighbor
relations for arbitrary Delone sets. More specifically, we present two
general and canonical ways of associating a neighbor relation to a
Delone set. These will be based on the  Voronoi construction and
the concept of tiling systems. We also show how suitable
decompositions of Euclidean space as a CW-complex  give rise to a
Delone set with a neighbor relation. By the considerations of the
previous section we then obtain in all these cases  both  a
combinatorial and a metric graph.

\bigskip

\begin{definition}[Tiling system] A pair $(\vL, (V_x)_{x\in \vL})$
consisting of a set $\vL \subset \RR^N$ and a family $V_x$, $x\in \vL$,
of compact subsets of $\RR^N$ is called a tiling system with
parameters $r,R>0$ if the following conditions hold:
\begin{itemize}
\item $ (x+ B_r) \subset V_x \subset (x + B_R)$ for all $x\in \vL$.
\item $\mbox{int} (V_x) \cap \mbox{int} (V_y) = \emptyset$ whenever $x,y
\in\vL$ with $x\neq y$. (Here, $\mbox{int}$ denotes the interior of
a set.)
\item $\bigcup_{x\in \vL} V_x = \RR^N$.
\end{itemize}
A set $V_x$ is referred to as  tile with distinguished  point $x$.
\end{definition}

By construction the set $\vL$ appearing in the definition of a
tiling system with parameters $r,R >0$  is an $(r,R)$-Delone set.

\smallskip

Tiling systems canonically  offer the possibility to define neighbor
relations:

\begin{example}[Canonical neighbor relation]
For each tiling system $(\vL, (V_x)_{x\in\vL})$ we can define a
neighbor relation (with parameter $2R$) by
$$\mathcal{N}_{can} (\vL, (V_x)_{x\in \vL}):=\{(x,y) \in \vL \times \vL : V_x \cap V_y\neq
\emptyset\}.$$ Indeed, (N0) is obviously satisfied and due to  $V_x
\subset (x + B_R)$ for any $x\in \vL$ the condition $(N1)$ follows.
Finally, one can easily infer (N2) by
 considering for any $x,y\in\vL$ with $x\neq y$ the set
 $$\{z \in \vL : V_z \cap [x,y] \neq \emptyset\}.$$
\end{example}

\begin{example}[Maximal neighbor relation]
For each tiling system $(\vL, (V_x)_{x\in\vL})$ we can define the
neighbor relation (with parameter $2R$)
$$\mathcal{N}_{max} (\vL, (V_x)_{x\in\vL}):=\{(x,y) \in \vL \times \vL : \|x - y\|\leq 2 R\}.$$
Clearly, $\mathcal{N}_{max} (\vL, (V_x)_{x\in\vL})$  contains the
canonical neighbor relation. Thus, it must satisfy (N2) as well.
Moreover, it is not hard to see that  (N0) and (N1) clearly satisfied. So,
$\mathcal{N}_{max} (\vL, (V_x)_{x\in\vL}) $ is indeed a neighbor
relation. It is the maximal neighbor relation with parameter $2R$ in
the sense
that any other such neighbor relation must be a subset of it.
\end{example}

\begin{example}[Tilings with convex polytopes]
If $(\vL,(V_x)_{x\in\vL})$ is a tiling system where the $V_x$ are
convex polytopes, we can also define the neighbor relation
$\mathcal{N}$ to consist of those $(x,y)$ such that $V_x$ and $V_y$
share a non-trivial part of an $(N-1)$-dimensional surface. It
requires some care to show that this is indeed a neighbor relation.
A proof can be given as follows: The properties (N0) and (N1) are
clear. Thus, it remains to show (N2).  Let $x,y\in \vL$ with $x\neq
y$ be given. Then, with $S_r:= B_r \setminus U_r$ the set
$$U:= \{ s\in (x + S_r): s\in [x,w] \mbox{ for some $w\in (y +U_r)$}
\}$$ is an open subset of $x + S_r$ and hence a positive surface
measure. On the other hand the set
$$N:=\{ s\in (x + S_r) : \mbox{the line through $x$ and $s$ meets
a $k$-dimensional surface of  of $V_z$}\}$$ has measure zero for any
$k\leq N-2$ and $z\in\vL$. Hence,  there must exist a $w\in (y +
U_r)$ such that $[x,w]$ does not intersect any $k$-dimensional face
of {$V_z$} for $k \leq (N-2)$ and {$z\in \vL$}. This rather directly
implies (N2).
\end{example}

The previous examples show that a Delone set admitting a tiling
system always gives rise to a neighbor relation as well. Next we show that
any Delone set appears in a  tiling system. This tiling system is
defined via the  Voronoi construction. As a consequence it is even a
tiling system consisting of convex polytopes.

We start with a discussion  of the well known Voronoi construction.
Let $\vL$ be an $(r,R)$-set. To an arbitrary $x\in \vL$ we associate
the  Voronoi cell $V(x,\vL) \subset \RR^N$ defined by
\begin{eqnarray*}
 V(x,\vL)&:=& \{p\in \RR^N : \|p-x\|\leq \|p-y\|\;\: \mbox{for all $y\in \vL$ with
$y\neq x$}\}\\
&=& \bigcap_{y\in \vL,y\neq x} \{p\in \RR^N :  \| p -x \|\leq \| p
-y\|\}.
\end{eqnarray*}
Note that $$\{p\in \RR^N : \| p -x \|\leq \| p -y\|\}$$ is a
half-space. Thus, $V(x,\vL)$ is a convex set. Moreover, it is
obviously closed and bounded and therefore compact. It turns out
that $V(x,\vL)$ is already determined by the elements of $\vL$ close
to $x$. More specifically, we have  $$V(x,\vL)= \bigcap_{y\in
B(x,2R)} \{p\in \RR^N : \| p -x \|\leq \| p -y\|\}$$ by  Corollary
5.2 in \cite{Sen}. This implies that the Voronoi cells are convex
polytopes. Moreover, also the following holds.

\begin{lemma} \label{voronoi} Let $\vL$ be an $(r,R)$-set and $x\in
\vL$ be arbitrary. Then, the following holds.

\begin{itemize}

\item[(a)] $V(x,\vL)$ is contained in $B(x,R)$.

\item[(b)] $V(x,\vL)$ contains $B(x,r)$.

\end{itemize}

\end{lemma}
\begin{proof} (a) This follows from  Proposition 5.2 in \cite{Sen}.

\smallskip

(b) This is immediate from the construction.
\end{proof}

The following proposition is a direct consequence of the preceding
lemma.

\begin{prop} Let $\vL$ be an $(r,R)$-set. Then, $(\vL,
(V_x)_{x\in\vL})$ is a tiling system with parameters $r,R$.
\end{prop}

By the previous proposition and the example with convex polytopes
discussed above, we infer that any Delone set $\vL$ comes with a
neighbor relation, the \textit{Voronoi neighbor relation},
$$\mathcal{N}_V (\vL) :=\{ (x,y)\in\vL\times \vL  :  \mbox{ $V_x$ and $V_y$
share a non-trivial part of an $(N-1)$-face}  \}.$$ Then, also the
bigger
$$\mathcal{N}_{max} (\vL) :=\{ (x,y)\in\vL\times \vL : \|x-y\| \leq 2R \}
$$
must be a neighbor
relation, called the \textit{maximal neighbor relation}.

\smallskip

We summarize the preceding considerations  in the following theorem.

\begin{theorem}[Existence of neighbor relations] Let $\vL$ be an $(r,R)$-Delone set. Then, both
$\mathcal{N}_V (\vL)$ and $\mathcal{N}_{max} (\vL)$ are neighbor
relations.
\end{theorem}


We finish this section by discussing a  related but slightly
different way of obtaining graphs from Delone sets via suitable
$CW$-complexes in Euclidean space.

\medskip

\begin{example}[CW-complexes in Euclidean space] Let a decomposition
of $\RR^N$ as a CW-complex be given such that the following
assumptions are satisfied:
\begin{itemize}

\item[  {(}C1)] Each cell is a convex polytope.
\item[(C2)] There exists a $\varrho>0$ such that each $N$-cell has
diameter at most $\varrho$.

\item[(C3)] There exists an $\gamma >0$ such that the $k$-dimensional surface  measure of any
$k$-cell is at least $\gamma$ for $k = 0,\ldots, N$.
\end{itemize}

Then, the $0$-skeleton $\vL$ of this complex is  $\gamma/2$ discrete
and $\varrho$ uniformly dense. Hence, it is a Delone set. We define
$$\mathcal{N}:=\{ (x,y) \in \vL\times \vL : \mbox{$x,y$ belong to
the same  $1$-dimensional cell}\}.$$ Then, $\mathcal{N}$ is a
neighbor relation with $S = \varrho$. Indeed, (N0) is clear and (N1)
follows as the length of a $1$-cell can not exceed the diameter of
any $N$-cell that it belongs to. It remains to show (N2): Let
$x,y\in \vL$ be given.  By (C2) any $N$-cell intersecting $[x,y]$ is
contained in $[x,y] + B_\varrho$. By construction any two points of
the $0$-skeleton belonging to the same $N$-cell are connected. Thus,
we infer (N2) with $S = \varrho$.
\end{example}

\section{Metric graphs over Delone sets and their
Laplacians}\label{sec-Metric} In this section we discuss the link
between metric graphs and the Dirichlet form of an associated
Laplacian. We will then also discuss some local regularity features
of the Dirichlet forms associated to metric graphs arising from
Delone sets. In this discussion we follow \cite{Hae} to which we
refer for further details and  proofs. We then go on and combine our
above considerations with \cite{BBK} to obtain the main result of
the paper.

\bigskip

Throughout this section we assume that we are given a Delone set
$\vL$ with a neighbor relation $\mathcal{N}$. We will be concerned
with the associated metric graph $(G_m (\vL, \mathcal{N}), \vL)$. In
our considerations we will need the Hilbert space $L^2 (G_m (\vL,
\mathcal{N}),\mu_m)$ of (equivalence classes of) real valued
functions on $G_m (\vL,\mathcal{N})$ with inner product
$$\langle u, v\rangle :=\int u v d\mu_m.$$

By construction each metric edge is isometrically isomorphic to an
open interval in $\RR$.  In the sequel we will then (tacitly)
identify the metric edge with this interval. This will allow us to
speak about weak differentiability and weak derivatives of functions
(by considering the corresponding functions on the open intervals).
In particular, whenever $u$ is a function on $G_m (\vL,\mathcal{N})$
which is weakly differentiable on each metric edge  we will denote
by $u'$ the derivative of $u$. Note that this derivative is not
defined on the  (countable many) branching points of  $G_m
(\vL,\mathcal{N})$.

By $W^{1,2} (G_m (\vL,\mathcal{N})  {)}$ we denote the vector space
of
 all continuous functions $u: G_m
(\vL,\mathcal{N})\longrightarrow \RR$ which are weakly
differentiable on each metric edge such that
$$\int |u(x)|^2  + |u' (x)|^2 d\mu_m  <\infty.$$
When equipped with the inner product
$$\langle u, v\rangle_{W^{1,2}}:=\int u v + u' v' d\mu_m$$
this becomes a Hilbert space.  The \textit{energy form} associated
to $G_m (\vL,\mathcal{N})$  is then given by $$ \CD:=\CD(\CE) =
W^{1,2} (G_m (\vL,\mathcal{N})), \;\: \CE(u,v) := \int u' v' d\mu_m
(x).$$ It is not hard to see that this is a closed symmetric form,
which is bounded below. The corresponding generator
$\Delta_{\vL,\mathcal{N}}$ is known as the \textit{Laplacian with
Kirchhoff boundary conditions} and for a function $u$ in the
operator domain, $\Delta_{\vL,\mathcal{N} }u=- u'' $ in a suitable
distributional sense, see e.g.  \cite{Hae} for further discussion.

The form $\CE$ is a
Dirichlet form, i.e.  whenever $u$ belongs to $\CD$ and $C :
\RR\longrightarrow\RR$ is a {\em normal contraction}
(meaning that $C$ satisfies $C(0) =0$ and $|C(x) - C(y)|\leq |x
- y|$ for all $x,y\in\RR$), then $C u$ belongs to $\CD$ as well and
$\CE (Cu, Cu)\leq \CE (u,u)$ holds.

The form has the regularity feature that $\CD \cap C_0 (G_m
(\vL,\mathcal{N}))$ is dense in  $\CD$ with respect to the form
norm. Indeed,
 chose an arbitrary  function $\varphi : \RR\longrightarrow [0,\infty)$,
 which is infinitely many times differentiable, supported in
 $[-2,2]$ and equal to $1$ in $[-1,1]$ and define
 $$\phi_n := \varphi (\frac{1}{n} d(p, \cdot))$$ for a fixed
 $p\in G_m (\vL,\mathcal{N})$. Then, a simple calculation shows that
 the $\phi_n u \in C_0(G_m(\Lambda,\mathcal{N}))$ converge to $u$ in
the form sense for any $u\in\CD$.

Thus, $\CE$ is a regular Dirichlet form.  Moreover,  $\CE$ is
strongly local in the sense that $\CE (u,v) = 0$ whenever $u$ is
constant on the support of $v$. In fact, the \textit{energy measure}
(which exists by abstract theory for any strongly local Dirichlet
form) can easily be seen to be given by
$$ d\Gamma(u(x)) = |u'(x)|^2 d\mu_m(x)$$
for $u\in\CD$ (cf. \cite{Hae}).

By the discussion in the previous section the graph $G_m
(\vL,\mathcal{N})$ has bounded geometry in the sense that the edge
lengths are uniformly bounded from below and the vertex degree is
uniformly bounded from above\footnote{By adding points on the edges one
can then automatically ensure that the edge lengths are bounded from
above as well.}. Hence, we can infer from \cite{Hae} the following
two local regularity properties, where we denote by $B_s (x)$ the
ball around $x$ with radius $s$ with respect to $d_m$.

\begin{prop}[Uniform local volume doubling]\label{VDl}
For any $L>0$ there  exists $\nu > 0$ such that for all $x\in G_m
(\vL,\mathcal{N})$ and all $s\in (0,L)$ we have
\[0< \mu_m(B_{2s}(x)) \leq 2^\nu \cdot \mu_m (B_s(x)) < \infty \label{VD}\tag{VD}.\]
\end{prop}

\begin{prop}[Uniform local Poincar\'{e} inequality]\label{PIl}
For any $L>0$  there exists $c_P
>0$ such that for all $0<s<L$ and $u\in W^{1,2} (B_s(x))$ we have
\[ \int\limits_{B_s(x)} |u(y) - \bar{u}_{B_s(x)}|^2 \: d  {\mu_m}(y) \leq c_P r^2 \int\limits_{B_s(x)} |u'(y)|^2 
\:d  {\mu_m}(y)\label{PI}\tag{PI}.\] Here, $\bar{u}_{B_s (x)}$ is
the average of $u$ over $B_s (x)$ given by
$$\bar{u}_{B_s (x)}:=\frac{1}{\mu_m (B_s (x))} \int_{B_s (x)} u
d\mu_m.$$

\end{prop}

Let now $e^{- t\Delta_{\vL,\mathcal{N}}}$, $t\geq 0$, be the
semigroup which is associated to the operator
$\Delta_{\vL,\mathcal{N}}$. Since $\CE$ is a Dirichlet form, this
semigroup is Markovian, i.e.
$$0 \leq e^{-t \Delta_{\vL,\mathcal{N}}} u  \leq 1$$
holds for any $u\in L^2 (G_m (\vL,\mathcal{N}))$ with $0 \leq u \leq
1$. Moreover,  in  \cite{Hae}  (see \cite{KLVW} as well) it was
shown that this semigroup has a kernel $p : (0,\infty) \times G_m
(\vL,\mathcal{N}) \times G_m (\vL,\mathcal{N})\longrightarrow
(0,\infty)$. Thus, we have
$$e^{-t \Delta_{\vL,\mathcal{N}}} u (x) = \int p_t (x,y) u (y)
d\mu_m (y)$$ for all $t>0$.

\begin{theorem}\label{thm-gaussian}
Let $\vL$ be a Delone set in $\RR^N$ with neighbor relation
$\mathcal{N}$ and $G_m (\vL,\mathcal{N})$ the associated  metric
graph. Let $\Delta_{\vL,\mathcal{N}}$ be the associated Laplacian.
Then, there exist $c_1, c_2,c_3, c_4>0$ such that the kernel $p$
satisfies
\[c_1 \frac{\exp(-c_2\frac{d(x,y)^2}{t})}{\mu_m (B_{\sqrt{t}} (x) ) } \leq p_t(x,y) \leq
c_3\frac{\exp(-c_4\frac{d(x,y)^2}{t})}{\mu_m(B_{\sqrt{t}}
(x))},\tag{GE}\label{gaussian}\] for all $x,y\in G_m
(\vL,\mathcal{N})$ and $t>0$.
\end{theorem}
\begin{proof} Given the local regularity properties discussed above
this follows from stability of corresponding heat kernel estimates
under rough isometries as discussed in \cite{BBK}. Here, the basic
idea behind this stability is that  volume doubling and Poincar\'{e}
inequality are stable under rough isometries. Now, by results of
Sturm \cite{St}   {v}olume doubling and Poincar\'{e} inequality are
equivalent to   {parabolic} Harnack inequality which in turns
implies the desired heat kernel estimates.  Here are the details for
our situation:

We note first that due to Corollary~\ref{Rough-isometry}, the space
$(G_m(\vL,\mathcal{N}),\mu_m,d_m)$ is roughly isometric to $(\RR^N,
\lambda,d)$. It follows from the Propositions~\ref{VDl}
and~\ref{PIl} that the underlying space $G_m(\vL, \mathcal{N})$
satisfies both the uniform local volume doubling property, as well
as the uniform local Poincar{\'e} inequality. This puts us in the
situation, where we can  apply Theorem~2.21~(a) from \cite{BBK} with
$X_1 = \mathbb{R}^N$ and $X_2 = G_m(\vL, \mathcal{N})$. Precisely,
we deduce from the latter result the full range volume doubling
property, as well as the global Poincar{\'e} inequality for $X_2$.
(Note that the statement in \cite{BBK} is stated in higher
generality. Our situation corresponds to the case $\beta_1 = \beta_2
= \overline{\beta} = 2$.) With this at hand, results from Sturm
\cite{St} on strongly local Dirichlet forms conclude the proof.
Indeed, noting that closed balls of finite $d_m$-radius are compact
(hypothesis~I(a) in \cite{St}), we have verified all assumptions
imposed in the just mentioned paper in order to obtain the desired
estimates on the heat kernel. Precisely, the Gaussian upper bound
follows from inequality~(4.4) in \cite{St} which was derived from
Theorem~4.1 of the same paper. The Gaussian lower bound is deduced
from Corollary~4.10 in \cite{St}. This finishes the proof.
\end{proof}

\begin{remark}
\begin{itemize}
\item The bounds on the heat kernel given in the preceding theorem
are known as \textit{Gaussian bounds}.

\item Note that no  further regularity conditions on the Delone set
are necessary. In particular, even Delone sets arising from corresponding
random processes will still satisfy the above Gaussian
bounds.

\item There is quite some stability to the argument. In fact, the
main ingredient is the rough isometry. This will also hold if, for
example,  the measure $\mu_m$ is replaced by the measure $h \mu$
with a measurable function $h : G_m (\vL,\mathcal{N})\longrightarrow
(0,\infty)$ being uniformly bounded away from zero and uniformly
bounded from above.

\item By applying this result to the metric graph with vertex set $\ZZ^2$ and edge set
$\{ \{tv+(1-t)w\}|\, 0 \leq t \leq 1 \hbox{ and }  v,w \in \ZZ^2,\,
|v-w|=1\}$ (which clearly fits into the framework of the example of
polytopal tilings) we recover the Gaussian bounds obtained by Pang
in \cite{Pang} by means of a completely different method. Note, however, that
the main thrust of \cite{Pang} is on the somewhat different issue of
scaling.

\item Our setup is based on Delone sets in Euclidean space (as this
is the most relevant situation from the point of view of
mathematical physics). However,  we would like to emphasize that
both  \cite{BBK} and \cite{St} are valid for rather general strongly
local Dirichlet forms. So, the above line of reasoning can be
generalized to substantially more general situations.

\end{itemize}
\end{remark}

\section{The discrete case}\label{sec-discrete}
In the preceding sections we have mainly been interested in the case
of metric graphs. However, one can also consider the discrete graph
associated to a Delone set and a neighbor relation. In this case,
one also obtains bounds on the corresponding semigroup by
essentially the same methods. This is discussed in this section.

\bigskip

Let $\vL$ be a Delone set and $\mathcal{N}$ a neighbor relation on
$\vL$. We consider the Hilbert space $\ell^2 (\vL) = L^2 (\vL,
\mu_c)$ consisting of square summable real valued functions on
$\vL$. The  discrete Laplacian $L_{\vL,\mathcal{N}}$ is the linear
operator defined on the whole $\ell^2 (\vL)$ via
$$L_{\vL,\mathcal{N}} u (x) :=\sum_{y\sim x} (u(x) - u(y)).$$
As the degree is uniformly bounded, this is a bounded operator. It
is a non-negative operator with form $\CE_c$ given by
$$\CE_c (u,v)= \frac{1}{2} \sum_{(x,y)\in \mathcal{N}} |u(x) -
u(y)|^2. $$ Clearly, this is a Dirichlet form  i.e. for any normal
contraction $C : \RR\longrightarrow \RR$ we have $\CE (Cu, Cu)\leq
\CE (u,u)$. Hence, the associated semigroup $e^{-t
L_{\vL,\mathcal{N} } }$ is Markovian i.e. satisfies $0 \leq e^{-t
L_{\vL,\mathcal{N}}} u \leq 1$ whenever $0\leq u \leq 1$ holds.
Since $\vL$ is a discrete set, the semigroup has a kernel $p :
(0,\infty)\times \vL\times \vL\longrightarrow (0,\infty)$.

\begin{theorem}\label{thm-gaussian-disc} Let $\vL$ be a Delone set in $\RR^N$ with neighbor relation
$\mathcal{N}$ and $G_c (\vL,\mathcal{N})$ the associated  discrete
graph. Let $L_{\vL,\mathcal{N}}$ be the associated Laplacian. Then,
there exist $c_1, c_2,c_3, c_4>0$ such that the kernel $p$ satisfies
\[c_1 \frac{\exp(-c_2\frac{d(x,y)^2}{t})}{\mu_c (B(x,\sqrt{t}))} \leq p_t(x,y) \leq
c_3\frac{\exp(-c_4\frac{d(x,y)^2}{t})}{\mu_c(B(x,\sqrt{t}))},\tag{GE}\label{gaussiandisc}\]
for all $x,y\in G_c (\vL,\mathcal{N})$ and $t>1\vee d_c (x,y)$.
\end{theorem}

\begin{proof} As the main result of the previous section, this
is essentially  a consequence  from Theorem 2.21 (and the remark
following it)  in \cite{BBK}. This time, however, we have to combine
this with considerations of Delmotte \cite{Del} and in particular
Theorem 3.8 of \cite{Del}. This theorem  shows the discrete time
version of the estimate~\eqref{gaussiandisc}. However, the line of
argumentation of \cite{Del} carries over to the continuous time case
as well.  For completeness reasons we next briefly  sketch the
corresponding argument.

We note that $(\RR^N,\lambda,d)$ and $(\vL,\mu_c, d_c)$ with   are
roughly isometric. Since the graph $G_c (\vL,\mathcal{N})$
 has uniformly bounded degree all local regularity requirements are
 automatically satisfied. Indeed, below some fixed scale $c_0\ge 1$,  local
 Poincar\'{e} inequality always holds for  $G_c (\vL,\mathcal{N})$. The argument
is the same as the proof of Proposition 5.4 (a) in \cite{BBK}. By
Theorem 2.21 (a) of  \cite{BBK} applied to $X_1 = (\RR^N,\lambda,d)$
and $X_2 = G_c (\vL,\mathcal{N})$, we obtain that $G_c
(\vL,\mathcal{N})$ satisfies volume doubling and Poincar\'{e}
inequality. The heat kernel estimates then follow from Theorem 3.8
of \cite{Del}.  Indeed, the on-diagonal estimates were shown in
Proposition 3.1 of  \cite{Del}. The Gaussian type upper bound in the
space time range  $t>1\vee d_c (x,y)$ can be concluded from a
combination of Proposition 3.4,  (3.18) and the first part of the
proof of Theorem 3.8 (applied to the continuous time setting) in
\cite{Del}. The corresponding Gaussian type lower bound follows from
the on-diagonal lower bound and the volume doubling condition via a
chaining argument (e.g.  the second part of the  proof of Theorem
3.8 in \cite{Del} applied to the continuous time setting).
\end{proof}

\begin{remark}
\begin{itemize}
\item As in the case of metric graphs, the bounds on the heat kernel given in the preceding theorem
are known as \textit{Gaussian bounds}.

\item As the argument is basically  the same as in the case of metric graphs,
there is, again, quite some stability. For example,  we can allow
for a weight function  $$b :\mathcal{N} \longrightarrow (0,\infty)$$
with $b(x,y) = b(y,x)$  on the edges provided these weights are
uniformly bounded from above and uniformly bounded away from zero.
Similarly, we could replace the measure $m_c$ by any measure $h m_c$
by any $h$ which is uniformly bounded away from zero and uniformly
bounded from above. The resulting operator then acts by
$$L u (x) = \frac{1}{h(x)} \sum_{y\sim x} b(x,y) ( u(x) - u(y)).$$

\item As a specific instance of the previous part of the remark, we may chose
$$h : \vL\longrightarrow (0,\infty), h(x) := |V_x| \mbox{ and }
b(x,y) := d(x,y)^l$$ for some $l\in\ZZ$, whenever
we are given a Delone set $\vL$ (and
consider it with the canonical neighbor relation).

\end{itemize}
\end{remark}

\bibliographystyle{amsalpha}

\end{document}